\documentclass[11pt,english]{amsart}
\usepackage{mathabx}
\usepackage{braket}
\usepackage{bbm}
\usepackage[T1]{fontenc}
\usepackage{amsfonts}
\usepackage{mathrsfs}
\usepackage{stmaryrd}
\usepackage{upgreek}
\usepackage{textcomp}
\usepackage[utf8]{inputenc}

\usepackage{extarrow}
\usepackage{enumerate}
\usepackage{enumitem}
\usepackage{braket}
\usepackage{graphicx}
\usepackage[colorlinks=true, linkcolor=blue]{hyperref}
\usepackage{latexsym}
\usepackage{mathtools}
\usepackage{mathtext}
\usepackage{amsmath}
\usepackage{amssymb}
\usepackage{amsthm}
\usepackage{tikz}
\usepackage{tikz-cd}
\usepackage[all,cmtip]{xy}
\usepackage[mathscr]{eucal}
\usepackage[toc,page]{appendix}
\usepackage{a4wide}
\usepackage[
backend=biber,
style=alphabetic-verb,
isbn=false, doi=false, url=false
]{biblatex}
\AtEveryBibitem{%
  \clearfield{addendum}%
  \clearfield{howpublished}%
}

\usepackage[a4paper,
  left=2.2cm,
  right=2.2cm,
  top=2.2cm,
  bottom=2.2cm]{geometry}
\usepackage{csquotes}
\MakeOuterQuote{"}

\usepackage[colorinlistoftodos]{todonotes}
\setuptodonotes{fancyline}
\usepackage{setspace} 

\usetikzlibrary{arrows}
\usetikzlibrary{matrix}
\usetikzlibrary{shapes}
\usetikzlibrary{snakes}
\usetikzlibrary{matrix}

\DeclareMathOperator{\et}{\textup{et}}

\usepackage{xcolor}
\definecolor{articlelink}{rgb}{0,0,0.45}
\hypersetup{linkcolor=articlelink,citecolor=articlelink,urlcolor=articlelink,
  pdftitle={Potential isotriviality of isocrystals and proper covers},
  pdfauthor={Adrian Langer and Lei Zhang}}
\theoremstyle{plain}
\newtheorem{thm}{Theorem}[section]
\newtheorem*{thm*}{Theorem}
\newtheorem{lem}[thm]{Lemma}
\newtheorem{cor}[thm]{Corollary}
\newtheorem{prop}[thm]{Proposition}

\theoremstyle{remark}
\newtheorem{rmk}[thm]{Remark}

\newtheorem*{rmk*}{Remark}
\newtheorem{ex}[thm]{Example}

\theoremstyle{definition}
\newtheorem{defn}{Definition}[section] 

\newtheorem*{const*}{Construction}
\newtheorem{conv}[defn]{Conventions}

\theoremstyle{plain}
\newtheorem{thmI}{Theorem}

\renewcommand{\et}{\textup{\'et}}

\newcommand{\Isoc}{{\rm Isoc}}

\newcommand{\Rep}{\text{\sf Rep}}

\AtEveryBibitem{\clearlist{language}}

\begin{document}
\title{Potential isotriviality of isocrystals and proper covers}

\author{Adrian Langer}
\email{alan@mimuw.edu.pl}
\address{University of Warsaw, Institute of Mathematics,
	ul.\ Banacha 2, 02-097 Warszawa, Poland}
	
\author{Lei Zhang}
\email{cumt559@gmail.com}
\address{Sun Yat-Sen University\\
    School of Mathematics (Zhuhai)\\    
    Zhuhai, 
    Guangdong Province\\ China}

\date{September 2026}

\begin{abstract}
We study convergent and overconvergent isocrystals that become trivial
after pullback along a proper surjective morphism.  On a geometrically
unibranch variety over an algebraically closed field, every such object is
already trivialized by a finite \'etale cover.  Over a perfect field, this
shows that a proper cover giving geometric triviality can be replaced by a
finite \'etale cover over the ground field, and implies finiteness of the
geometric monodromy group.  For proper geometrically unibranch
varieties, finite geometric
monodromy is also sufficient.  In the overconvergent case, a dominant
morphism can replace the given proper cover.  A nodal curve shows why the
geometrically unibranch hypothesis is needed.
\end{abstract}

\maketitle

\section*{Introduction}
Let $E$ be a convergent or overconvergent isocrystal on a variety $X$
over a perfect field of positive characteristic.  We study the relation
between trivializability by a proper surjection, trivializability by a
finite \'etale cover, and finiteness of geometric monodromy.  No Frobenius
structure is assumed.

Proposition~\ref{prop:proper-trivialization} gives the main geometric
result.  If the ground field is algebraically closed and $X$ is connected
and geometrically unibranch, every proper trivialization can be replaced
by a finite \'etale one.  No properness assumption on $X$ is needed.
The proof uses Stein factorization and effective descent to obtain a finite
cover.  After passing to a Galois closure, the inertia groups act trivially
on horizontal sections, so the representation factors through an \'etale
quotient.

We call $E$ \emph{potentially isotrivial} if its geometric monodromy
group is finite (Section~\ref{sec:conventions}).  Combining the proposition
with Lazda--P\'al's theorem for smooth projective varieties
\cite[Thm.~13.4]{LP17} gives the following criterion.

\begin{samepage}
\begin{thmI}\label{main}
Let $X$ be a proper, geometrically connected, geometrically unibranch
variety over a perfect field $k$ of characteristic $p>0$, and let
$E$ be a convergent or overconvergent isocrystal on $X/K$.
Then the following conditions are equivalent:
\begin{enumerate}
\item the geometric monodromy group of $E$ is finite;
\item a finite \'etale cover of $X_{\overline k}$ trivializes
      $E_{\overline k}$;
\item there is a finite \'etale surjection $v\colon Y\to X$ such that
      $v^*E$ is geometrically trivial;
\item there is a proper surjective morphism $f\colon Y\to X$ such that
      $f^*E$ is geometrically trivial.
\end{enumerate}
\end{thmI}
\end{samepage}

For a possibly nonproper $X$, conditions (2)--(4) remain equivalent and
imply (1)
(Theorem~\ref{thm:proper-cover-geometric-triviality}).  The reverse
implication is a finite-\'etale realization question; see
Remark~\ref{rmk:finite-monodromy-realization}.  For overconvergent
isocrystals, full faithfulness of restriction to a dense open subset
allows us to replace a proper surjection by any dominant morphism of
varieties (Corollary~\ref{cor:dominant-cover-potential-isotrivial}).

Related trivialization questions for vector bundles were studied by
Biswas--dos Santos \cite{BDS11,BDS12} and Tonini--Zhang \cite{TZ3}.
For $F$-divided sheaves, Tonini--Zhang \cite{TZ2} prove essential
finiteness under dominant trivializations.  The results below concern
convergent and overconvergent isocrystals with coefficients in a
mixed-characteristic field.

Finally,
Example~\ref{ex:nodal-potential-isotrivial-descent} exhibits a nontrivial
rank-one object on a nodal curve whose normalization trivializes it.
Its infinite monodromy shows that the unibranch hypothesis cannot be
omitted. This example also appears in \cite[Ex.~7.1]{Tsu03}, where it is used to show that the restriction map in rigid cohomology to a dense open subset need not be surjective when the ambient variety is not geometrically unibranch.

\section{Conventions and geometric monodromy}
\label{sec:conventions}

Throughout, $k$ is a perfect field of characteristic $p>0$, and
$V$ is a complete discrete valuation ring of mixed characteristic
$(0,p)$, with residue field $k$ and fraction field $K$.
A \emph{variety} is a separated $k$-scheme of finite type; it need not
be reduced or irreducible.  A proper cover means a proper surjective
morphism to the indicated base.

We write $\Isoc^{\dagger}(X/K)$ for convergent isocrystals and
$\Isoc^{\dagger\dagger}(X/K)$ for overconvergent isocrystals, without a
chosen Frobenius structure.  These are $K$-linear abelian rigid tensor
categories.  Their objects have locally constant finite rank, and
pullback is exact and preserves rank; see
\cite[Prop.~1.3.3 and Sections~2.3.3(ii)--(iv), 2.3.4]{Berthelot96}.
Here Proposition~1.3.3 gives connectedness of tubes, and Section~2.3.4
compares the convergent category with Ogus's definition.  The flatness
condition in Section~2.3.3(iv) concerns the coefficient base, which is
fixed here.  Evaluation at a $k$-rational point $x$ is denoted by
$\omega_x$; when $X$ is connected, Lemma~\ref{lem:pullback-faithful}
and the observation following it show that this is a neutral fibre functor.
We use
$\langle E\rangle^\otimes$ for the strictly full subcategory generated
by $E$ under finite direct sums, tensor products, duals and
subquotients.  An object of rank $r$ is \emph{trivial} if it is
isomorphic to $\mathbf 1_X^{\oplus r}$.

For $\star\in\{\dagger,\dagger\dagger\}$, put
\[
    \mathscr C^\star(X)\coloneqq \Isoc^\star(X/K).
\]
Fix an algebraic closure $\overline k$ of $k$, and let
$V_{\overline k}=\widehat{V^{\mathrm{nr}}}$ and
$K_{\overline k}=\operatorname{Frac}(V_{\overline k})$, where the
completion is taken after choosing a maximal unramified extension with
residue field $\overline k$.
For a geometrically connected $k$-variety $X$, write
\[
  \overline X=X\times_k\overline k,
  \qquad
  \overline E=E_{\overline k}
      \in\Isoc^\star(\overline X/K_{\overline k}).
\]
Here $E_{\overline k}$ denotes the image of $E$ under simultaneous
base change of the residue and coefficient fields (with the corresponding completion on coefficient rings).
After choosing a geometric point $\overline x\in\overline X(\overline k)$, the
\emph{geometric monodromy group} of $E$ is
\[
 G_{\mathrm{geom}}(E,\overline x)
 =\operatorname{Aut}^{\otimes}
   \bigl(\omega_{\overline x}|_{\langle\overline E\rangle^\otimes}\bigr).
\]
Whether this affine $K_{\overline k}$-group scheme is finite or trivial is
independent of the chosen geometric point (and of a further faithful scalar
extension).  We call $E$
\emph{potentially isotrivial} if
$G_{\mathrm{geom}}(E,\overline x)$ is finite, and
\emph{geometrically trivial} if this group is trivial.  The latter condition
is equivalent to
\[
       \overline E\simeq
       \mathbf 1_{\overline X}^{\oplus\operatorname{rk}(E)}.
\]
For a scheme whose geometric base change is disconnected, geometric
triviality means triviality on every connected component.  The statements
below can likewise be read componentwise if $X$ is not geometrically
connected.

\begin{lem}[Faithfulness of pullback]
\label{lem:pullback-faithful}
Let $f\colon Y\to X$ be a morphism of $k$-varieties whose image meets
every connected component of $X$.  Then pullback on convergent or
overconvergent isocrystals is faithful.
\end{lem}
\begin{proof}
Let $u\colon E\to F$ be a morphism with $f^*u=0$.  Exactness gives
$f^*(\operatorname{im}u)=0$.  Since pullback preserves rank, the object
$\operatorname{im}u$ has rank zero at a point of every connected
component of $X$.  Its rank is locally constant, so
$\operatorname{im}u=0$ on $X$, and hence $u=0$.
\end{proof}

In particular, if $X$ is connected and $x\in X(k)$, then
\[
 x^*\colon\operatorname{End}(\mathbf 1_X)
       \longrightarrow
       \operatorname{End}(\mathbf 1_{\operatorname{Spec}k})=K
\]
is injective by the lemma.  Its image contains the scalar endomorphisms,
so it is an isomorphism.  Evaluation at $x$ is an exact faithful
$K$-linear tensor functor to finite-dimensional $K$-vector spaces;
thus the category is neutral Tannakian.  This applies, in particular, to
every nonempty connected variety over an algebraically closed field.

\section{Proper trivializations}
\label{sec:proper-trivializations}

The following proposition is the geometric input for the monodromy
statements.  Its proof applies to both coefficient categories.

\begin{prop}[Proper trivializations over an algebraically closed field]
\label{prop:proper-trivialization}
Assume that $k$ is algebraically closed.  Let $X$ be a connected
geometrically unibranch $k$-variety, let
$\star\in\{\dagger,\dagger\dagger\}$, and let
$E\in\mathscr C^\star(X)$.  If there are a $k$-scheme $Y$ and a
proper surjective morphism $f\colon Y\to X$ such that
\[
\refstepcounter{equation}
       f^*E\simeq\mathbf 1_Y^{\oplus n},
       \qquad n=\operatorname{rk}(E),
       \label{eq:proper-trivialization}
\tag{\theequation}
\]
then some finite \'etale surjection $z\colon Z\to X$ trivializes $E$.
\end{prop}

\begin{proof}
We use three standard facts.  Proper surjections are morphisms of effective
descent for convergent and overconvergent isocrystals
\cite[Thm.~4.6]{Ogus84} and \cite[Thm.~5.1]{Laz22}; pullback from a connected
variety to a nonempty variety is faithful by
Lemma~\ref{lem:pullback-faithful}; and pullback by a finite universal
homeomorphism is an equivalence
\cite[Cor.~4.10]{Ogus84} and \cite[Thm.~6.1]{Laz22}.

We first reduce to the normal integral case.  Passing to
$X_{\mathrm{red}}$ and then to its normalization $X^\nu$ does not change
either the category under consideration or the category of finite \'etale
covers: varieties over a field are excellent, so the normalization is
finite, and the relevant maps are finite universal homeomorphisms
\cite[\href{https://stacks.math.columbia.edu/tag/0GIQ}{0GIQ}]{stacks-project}, and finite \'etale covers
are invariant under universal homeomorphisms
\cite[\href{https://stacks.math.columbia.edu/tag/0BQN}{0BQN}]{stacks-project}.  Here we use that a
connected geometrically unibranch variety has irreducible reduction.
Replace $X$ by $X^\nu$ and $Y$ by $Y\times_X X^\nu$; then choose an
irreducible component of the latter dominating $X^\nu$, and replace it by
its reduction and normalization.  Properness shows that it still surjects,
and \eqref{eq:proper-trivialization} is preserved.  We may therefore assume
that both $X$ and $Y$ are normal and integral.  At the end, invariance
of finite \'etale covers and the equivalences on isocrystals transport both
the resulting cover and its trivialization uniquely back to the original
$X$.

Consider the Stein factorization
\[
       Y\xrightarrow{q}S\xrightarrow{p}X.
\]
The morphism $q$ is proper and surjective with geometrically connected
fibres, whereas $p$ is finite and surjective
\cite[\href{https://stacks.math.columbia.edu/tag/03H0}{03H0}]{stacks-project}.  Put $F=p^*E$ and fix
an isomorphism
\[
       \theta\colon q^*F\xrightarrow{\sim}\mathbf 1_Y^{\oplus n}.
\]
The scheme $R=Y\times_S Y$ is connected: the first projection $R\to Y$
is proper and surjective with geometrically connected fibres, and a
nontrivial open-and-closed decomposition of $R$ would therefore induce
one of $Y$.  Transporting the canonical descent datum on $q^*F$ through
$\theta$ gives an automorphism
\[
       \alpha\colon\mathbf 1_R^{\oplus n}
                    \xrightarrow{\sim}\mathbf 1_R^{\oplus n}.
\]
Its pullback along the diagonal $\Delta\colon Y\to R$ is the identity.
Since $R$ is connected and $Y$ is nonempty, Lemma
\ref{lem:pullback-faithful} says that $\Delta^*$ is faithful.  Hence
$\alpha=\mathrm{id}$, so $\theta$ is compatible with the descent datum.
Effective descent along $q$ now yields
\[
\refstepcounter{equation}
       p^*E\simeq\mathbf 1_S^{\oplus n}.
       \label{eq:finite-trivialization}
\tag{\theequation}
\]

It remains to replace the finite cover in
\eqref{eq:finite-trivialization} by a finite \'etale cover.  Write
\[
       K_X=k(X),\qquad L=k(S),
\]
and let $L_{\mathrm{sep}}$ be the maximal separable subextension of
$L/K_X$.  If $S_{\mathrm{sep}}$ is the normalization of $X$ in
$L_{\mathrm{sep}}$, then $S\to S_{\mathrm{sep}}$ is a finite universal
homeomorphism.  Topological invariance shows that $E|_{S_{\mathrm{sep}}}$
is trivial.  We may thus assume that $L/K_X$ is separable.

Choose a finite Galois extension $M/K_X$ containing $L$, let $T$ be
the normalization of $X$ in $M$, and set
\[
       G=\operatorname{Gal}(M/K_X),\qquad \pi\colon T\to X.
\]
Then $\pi^*E$ is trivial.  Put
\[
 W=\operatorname{Hom}_{\mathscr C^\star(T)}(\mathbf 1_T,\pi^*E).
\]
Since $k$ is algebraically closed and $T$ is connected,
$\operatorname{End}(\mathbf 1_T)=K$, and $W\simeq K^{\oplus n}$.
The canonical $G$-linearization of $\pi^*E$ induces a $K$-linear
action
\[
       \rho\colon G\longrightarrow\operatorname{GL}_K(W).
\]
Let $H=\ker(\rho)$; it is a normal subgroup of $G$.

Every inertia group of $T/X$ lies in $H$.  Indeed, if $g$ belongs to
an inertia group, its scheme-theoretic fixed locus on $T$ is nonempty;
choose a closed point $t$ in that locus.  Since $k$ is algebraically
closed, $\kappa(t)=k$, and $g$ acts trivially on the residue field.  For
each $v\in W$, the canonical linearization therefore gives
\[
       t^*(\rho(g)v)=t^*v.
\]
Pullback to $t$ is faithful by Lemma~\ref{lem:pullback-faithful}, so
$\rho(g)v=v$ for every $v$, as asserted.

Since $H$ is normal, $M^H/K_X$ is Galois.  Let $Z$ be the
normalization of $X$ in $M^H$; then $Z=T/H$, and
\[
       T\xrightarrow{s}Z\xrightarrow{z}X.
\]
The inertia group of $Z/X$ at any point is a quotient of the corresponding
inertia group of $T/X$
\cite[\href{https://stacks.math.columbia.edu/tag/0BSX}{0BSX}]{stacks-project}.  It is therefore trivial,
and $z$ is finite \'etale
\cite[\href{https://stacks.math.columbia.edu/tag/0BTF}{0BTF}]{stacks-project}.

Finally, a $K$-basis of $W$ gives an $H$-equivariant isomorphism
\[
       \psi\colon\mathbf 1_T^{\oplus n}\xrightarrow{\sim}\pi^*E.
\]
Because $s$ can be ramified, equivariance on the graphs of elements of
$H$ is not by itself the full descent condition.  Let
$R_H=T\times_ZT$, and transport the canonical descent datum through
$\psi$, obtaining
\[
       \beta\in
       \operatorname{Aut}_{\mathscr C^\star(R_H)}
       (\mathbf 1_{R_H}^{\oplus n}).
\]
Its restriction to the graph $\Gamma_h$ is the identity for every
$h\in H$.  Since $Z=T/H$, the map
\[
       \coprod_{h\in H}T\longrightarrow T\times_ZT
\]
is surjective on underlying topological spaces: two geometric points of
$T$ have the same image in the finite quotient precisely when they belong
to the same $H$-orbit.  Pullback along this morphism is faithful by
Lemma~\ref{lem:pullback-faithful}.  Since $\beta$ restricts to the identity
on each graph, it follows that $\beta=\mathrm{id}$.  Thus $\psi$ satisfies the full \v Cech descent
condition.  Effective descent along the finite surjection $s$ yields
\[
       z^*E\simeq\mathbf 1_Z^{\oplus n}.
\]
The finite \'etale cover $z$ trivializes $E$, completing the proof.
\end{proof}

\begin{rmk}[Quotient stacks]
\label{rmk:quotient-stacks}
The final descent step can be expressed in terms of quotient stacks.
Convergent isocrystals on algebraic stacks are defined in
\cite[Section~2.1]{Shiho2010}; Example~2.3 identifies the category on a
quotient by a finite \'etale group with equivariant isocrystals on its
atlas.  For overconvergent isocrystals, see \cite{Brown2010}; for the
finite quotients used here, equivariant objects and \'etale effective
descent suffice.

In the notation of the proof, let $\mathcal Z=[T/H]$ and
$c\colon\mathcal Z\to Z=T/H$.  The map $T\to\mathcal Z$ is a finite
\'etale $H$-torsor, even when $T\to Z$ is ramified.  Since $H$
acts trivially on $W$, a basis of $W$ trivializes the pullback of
$z^*E$ to $\mathcal Z$.  The graph-cover argument in the proof also
shows that $c^*$ is fully faithful: an $H$-equivariant morphism
between objects pulled back from $Z$ satisfies the full \v Cech
condition after pullback along
\[
 \coprod_{h\in H}T\longrightarrow T\times_ZT,
\]
and Lemma~\ref{lem:pullback-faithful} detects this equality before
pullback.  Effective descent then gives a unique morphism on $Z$.
The trivialization and its inverse therefore descend to $Z$.

This is analogous to the quotient-stack argument of
\cite[Sections~3.6--3.8]{TZ3}.  The inertia calculation still supplies
the essential geometric conclusion that $Z\to X$ is finite \'etale:
the morphism $\mathcal Z\to X$ itself need not be representable.
\end{rmk}

\section{Proper covers and finite geometric monodromy}
\label{sec:potential-isotriviality}

\begin{thm}[Proper covers and geometric triviality]
\label{thm:proper-cover-geometric-triviality}
Let $X$ be a geometrically connected, geometrically unibranch
$k$-variety, let $\star\in\{\dagger,\dagger\dagger\}$, and let
$E\in\mathscr C^\star(X)$.  The following conditions are equivalent:
\begin{enumerate}
    \item there is a finite \'etale surjection
    $u\colon X'\to\overline X$ such that $u^*\overline E$ is trivial;
    \item there is a finite \'etale surjection $v\colon Y\to X$ such that
    $v^*E$ is geometrically trivial;
    \item there are a $k$-scheme $Y$ and a proper surjective morphism
    $f\colon Y\to X$ such that $f^*E$ is geometrically trivial.
\end{enumerate}
Under these conditions, $G_{\mathrm{geom}}(E,\overline x)$ is a constant
finite group scheme over $K_{\overline k}$.  In particular, $E$ is potentially isotrivial. 
\end{thm}

\begin{proof}
We first prove $(1)\Rightarrow(2)$.  Let
$u\colon X'\to\overline X$ be a finite \'etale cover
trivializing $\overline E$.  Since $\overline X$ is connected, any
connected component of $X'$ still surjects onto $\overline X$, so we
may assume that $X'$ is connected.  After replacing a finite extension
$k'/k$ by its Galois closure, the cover descends to a finite \'etale cover
$u_0\colon X'_0\to X_{k'}$.  Regard $X'_0$ as a $k$-scheme and
let
\[
       v\colon Y=X'_0\xrightarrow{u_0}X_{k'}\longrightarrow X.
\]
Since $k$ is perfect, $v$ is finite \'etale and surjective.  The connected components of $Y_{\overline k}$ are the
conjugates of $X'$ indexed by the $k$-embeddings
$k'\hookrightarrow\overline k$.  Extend each such embedding to an
element of $\operatorname{Gal}(\overline k/k)$.  Its semilinear action
on $K_{\overline k}$, together with the descent of $\overline E$ to
$E$, carries the given trivialization to one on the corresponding
component. Thus $v^*E$ is geometrically trivial.

The implication $(2)\Rightarrow(3)$ is trivial.  To prove $(3)\Rightarrow(1)$, base change $f$ to
$\overline k$.  The morphism
$\overline f\colon Y_{\overline k}\to\overline X$ remains proper and
surjective, and $\overline f^*\overline E$ is trivial on every connected
component.  Proposition~\ref{prop:proper-trivialization}, applied after
replacing $(k,K)$ by $(\overline k,K_{\overline k})$ and choosing a
component dominating $\overline X$, gives a finite \'etale cover of
$\overline X$ trivializing $\overline E$.

We now prove the assertion about monodromy.  Put
$L=K_{\overline k}$.  A connected component of a finite \'etale
trivializing cover is still surjective and can be dominated by a connected
finite \'etale Galois cover
\[
        q\colon T\longrightarrow\overline X,
        \qquad G=\operatorname{Aut}_{\overline X}(T).
\]
Let $\mathscr C_q$ be the full subcategory of
$\Isoc^\star(\overline X/L)$ consisting of objects whose pullback to
$T$ is trivial.  Since $T$ is connected and $\overline k$ is
algebraically closed, $\operatorname{End}(\mathbf 1_T)=L$.  For
$M\in\mathscr C_q$, set
\[
        W_M=\operatorname{Hom}_{\Isoc^\star(T/L)}
                    (\mathbf 1_T,q^*M).
\]
The evaluation morphism
\[
        \mathbf 1_T\otimes_L W_M\xrightarrow{\sim}q^*M
\]
is an isomorphism.  The canonical $G$-linearization of $q^*M$ gives
an $L$-linear action of $G$ on $W_M$.

Conversely, a representation $W$ of $G$ equips
$\mathbf 1_T\otimes_L W$ with a descent datum along $q$.  Indeed,
\[
        T\times_{\overline X}T\simeq\coprod_{g\in G}T,
\]
and the cocycle condition on this disjoint union is exactly the group law
for the action on $W$.  Effective descent produces an object of
$\mathscr C_q$.  Morphisms descend precisely when the corresponding
linear maps are $G$-equivariant.  These constructions are inverse tensor
equivalences, giving
\[
\refstepcounter{equation}
        \mathscr C_q\simeq\operatorname{Rep}_L(G).
        \label{eq:galois-descent-representations}
\tag{\theequation}
\]
The subcategory $\mathscr C_q$ is closed under subobjects and quotients
in $\Isoc^\star(\overline X/L)$: pullback is exact, and a subquotient of
a trivial object in a Tannakian category is again trivial.  It is also
closed under finite direct sums, tensor products and duals.  Hence
$\langle\overline E\rangle^\otimes$ corresponds under
\eqref{eq:galois-descent-representations} to the full Tannakian subcategory
generated by the representation
\[
        \rho\colon G\longrightarrow\operatorname{GL}_L(W_{\overline E}).
\]
Choose $t\in T(\overline k)$ above $\overline x$.  Evaluation at
$t$ identifies $W_M$ with $\omega_{\overline x}(M)$, compatibly
with tensor products.  The inclusion of this full Tannakian subcategory
therefore gives a faithfully flat homomorphism
\[
        G_L\longrightarrow G_{\mathrm{geom}}(E,\overline x)
\]
by \cite[Prop.~2.21(a)]{DM82}, where $G_L$ is the constant
finite group scheme over $L$.  Its kernel is
$(\ker\rho)_L$, since $\overline E$ is a tensor generator.  Thus
\[
        G_{\mathrm{geom}}(E,\overline x)
             \simeq (G/\ker\rho)_L
\]
is a constant finite group scheme.
This shows that $E$ is potentially isotrivial.
\end{proof}

For proper varieties, finite geometric monodromy is also sufficient.
Lazda--P\'al prove finite \'etale trivializability in the smooth
projective case.  An alteration and
Proposition~\ref{prop:proper-trivialization} extend their result to proper,
geometrically connected, geometrically unibranch varieties.

\begin{thm}[Proper-cover criterion for potential isotriviality]
\label{thm:proper-cover-potential-isotrivial}
Let $X$ be a proper, geometrically connected, geometrically unibranch
$k$-variety,
let $\star\in\{\dagger,\dagger\dagger\}$, and let
$E\in\mathscr C^\star(X)$.  Then the following are equivalent:
\begin{enumerate}
    \item $E$ is potentially isotrivial;
    \item there are a $k$-scheme $Y$ and a proper surjective morphism
    $f\colon Y\to X$ such that $f^*E$ is geometrically trivial.
\end{enumerate}
\end{thm}

\begin{proof}
Suppose that $E$ is potentially isotrivial.  Work with
$(\overline X,\overline E)$ over the coefficient pair
$(\overline k,K_{\overline k})$, and first replace $\overline X$ by its
reduction; this is harmless by topological invariance, and the reduction is
integral.  De Jong's theorem \cite[Thm.~4.1]{dJ96} gives a connected
alteration
\[
             a\colon X'\longrightarrow\overline X
\]
and a dense open immersion of $X'$ into a regular projective variety.
Since $a$ and $\overline X$ are proper, $X'$ is proper over
$\overline k$; this dense open immersion is therefore also closed and is
an isomorphism.  Thus $X'$ is projective and, since $\overline k$ is
perfect, smooth.  Choose
$\overline x'\in X'(\overline k)$ and put
$\overline x=a(\overline x')$.  The monodromy group of
$a^*\overline E$ is the scheme-theoretic image of the induced map into
$G_{\mathrm{geom}}(E,\overline x)$; it is therefore a closed finite
subgroup.  By
\cite[Thm.~13.4]{LP17}, applied over
$(\overline k,K_{\overline k})$, a finite \'etale cover
$b\colon Z\to X'$ trivializes $a^*\overline E$.  Lazda--P\'al's theorem
applies to both categories appearing here, since convergent and
overconvergent isocrystals coincide on the smooth projective variety $X'$:
with empty boundary the overconvergence condition imposes no additional
requirement \cite[Sections~2.3.2(i) and 2.3.6]{Berthelot96}.
The composite $a\circ b\colon Z\to\overline X$ is proper and surjective,
so Proposition~\ref{prop:proper-trivialization} supplies a finite \'etale
cover of $\overline X$ trivializing $\overline E$.  Theorem
\ref{thm:proper-cover-geometric-triviality} now gives (2).  The reverse
implication follows from the same theorem.
\end{proof}

\begin{rmk}[The finite-\'etale realization question]
\label{rmk:finite-monodromy-realization}
For a general geometrically unibranch $X$, Theorem
\ref{thm:proper-cover-geometric-triviality} proves that the analogue of
condition~(2) in Theorem
\ref{thm:proper-cover-potential-isotrivial} implies potential
isotriviality.  The converse requires a separate realization statement: an abstract
finite Tannaka group does not, by itself, construct a finite \'etale cover
of the underlying variety.  Indeed, for overconvergent
isocrystals without Frobenius, the question whether finite monodromy implies
finite \'etale trivializability is explicitly raised in
\cite[Rmk.~7.26]{Pal2022}.  By Theorem
\ref{thm:proper-cover-geometric-triviality}, this finite-\'etale realization
question is precisely the converse implication in the nonproper case.
We do not assert that converse for arbitrary nonproper $X$.
\end{rmk}

For overconvergent isocrystals, relative properness of the given cover can
also be omitted in the implication towards potential isotriviality.

\begin{cor}[Dominant covers in the overconvergent case]
\label{cor:dominant-cover-potential-isotrivial}
Let $X$ be a geometrically connected, geometrically unibranch
$k$-variety, and let $E\in\Isoc^{\dagger\dagger}(X/K)$.  If there is a
dominant morphism of $k$-varieties $f\colon Y\to X$ such that $f^*E$
is geometrically trivial, then $\overline E$ is trivialized by a finite
\'etale cover of $\overline X$.  In particular, $E$ is potentially
isotrivial.
\end{cor}

\begin{proof}
Work over $\overline k$.  Since $\overline X$ is connected and
geometrically unibranch, its reduction is integral.  Choose a reduced
irreducible component $Y_0$ of $Y_{\overline k}$ dominating
$\overline X$, and let $Y_1\to Y_0$ be its normalization.  The
pullback of $\overline E$ to $Y_1$ is trivial.

Apply Nagata compactification to $Y_1\to\overline X$, take the reduced
closure of $Y_1$ in the compactification, and normalize.  Since $Y_1$
is normal, this gives a dense open immersion $j\colon Y_1\hookrightarrow P$
into a normal integral variety and a proper morphism
$g\colon P\to\overline X$ extending $Y_1\to\overline X$.
The morphism $g$ is dominant, hence surjective.

The overconvergent isocrystal $g^*\overline E$ is trivial on $Y_1$.
Restriction from $P$ to $Y_1$ is fully faithful by
\cite[Cor.~1.2]{Tsu12}, since $P$ is normal and $Y_1$ is dense.
To apply the cited statement for pairs, choose a compactification of
$P$ over $\overline k$ and use it as a common compactification of $P$
and $Y_1$.  The trivialization and its inverse therefore extend uniquely
to $P$; faithfulness shows that their compositions are identities.
Thus $g^*\overline E$ is trivial, and
Proposition~\ref{prop:proper-trivialization} applied to
$g\colon P\to\overline X$ proves the claim.
\end{proof}

Let $X$ be a proper, connected, geometrically unibranch variety over
an \emph{algebraically closed field} $k$ of positive characteristic, and let
	$x\in X(k)$. For $\star\in\{\dagger,\dagger\dagger\}$, let  $\pi_1^{\star}(X/K, x)$
be the fundamental group scheme corresponding to $\operatorname{Isoc}^{\star}(X/K)$ with the fiber functor defined by $x$.
The following corollary is a generalization of \cite[Thm.~13.3]{LP17} that deals with smooth projective varieties. This is also a variant of \cite[Prop.~4.4]{Crew92}.

\begin{cor}\label{cor:isocrystal-component-group}
	 	The profinite completion \(\pi_0(\pi_1^{\star}(X/K, x))\) of the affine $K$-group scheme 
	\(
    \pi_1^{\star}(X/K, x)\) is the 
    pro-constant
	$K$-group scheme associated with
	$\pi_1^{\et}(X,x)$.
\end{cor}

\begin{proof}
	Let $\mathcal C_{\mathrm{fin}}$ be the full Tannakian subcategory of
	$\operatorname{Isoc}^{\star}(X/K)$ consisting of objects with finite
    monodromy. Let $\Rep_K^{\sc fq}(\pi_1^{\et}(X,x))$ denote the category of finite-dimensional \(K\)-representations factoring through a finite quotient.  
    Then equivalence \eqref{eq:galois-descent-representations} defines a fully faithful functor $\lambda\colon\Rep_K^{\sc fq}(\pi_1^{\et}(X,x))\to \mathcal C_{\mathrm{fin}}$. Indeed, for any Galois cover $q\colon Y\to X$ with Galois group $G$, \eqref{eq:galois-descent-representations} defines a natural equivalence $\Rep_K(G)\simeq \mathcal{C}_q$, where $\mathcal{C}_q\subset\mathcal{C}_{\mathrm{fin}}$ is the full subcategory consisting of objects that are trivialized by $q$. These equivalences, taken over pointed connected finite \'etale Galois covers, are compatible with refinement and define \(\lambda\) by passage to the filtered colimit. Theorem \ref{thm:proper-cover-potential-isotrivial} and Proposition \ref{prop:proper-trivialization} then show that $\lambda$ is essentially surjective. This completes the proof.
\end{proof}

\begin{ex}[The unibranch hypothesis is necessary]
\label{ex:nodal-potential-isotrivial-descent}
Assume that $k$ is algebraically closed, and let $X$ be the rational
projective nodal curve obtained from $\mathbb P^1_k$ by identifying
$0$ and $1$.  Its normalization
\[
       \nu\colon\mathbb P^1_k\longrightarrow X
\]
is finite and surjective.  Scheme-theoretically,
\[
 \mathbb P^1_k\times_X\mathbb P^1_k
 =\Delta(\mathbb P^1_k)\amalg\{(0,1)\}\amalg\{(1,0)\}.
\]
Indeed, at the node the completed local ring and its normalization are
\[
 A=k[[u,v]]/(uv),\qquad B=k[[u]]\times k[[v]].
\]
The same-branch tensor products give the two branches of the diagonal,
whereas each cross-branch tensor product is $k$.  Hence
\[
 B\otimes_A B\simeq B\times k\times k.
\]
Together with the isomorphism away from the node, faithful flatness of
completion verifies the displayed scheme-theoretic decomposition.

Choose $\lambda\in K^\times$ of infinite order.  On
$\mathbf 1_{\mathbb P^1}$, take the descent datum which is the identity
on the diagonal, multiplication by $\lambda$ on $(0,1)$, and
multiplication by $\lambda^{-1}$ on $(1,0)$.  It satisfies the cocycle
condition, so proper effective descent produces a rank-one convergent,
respectively overconvergent, isocrystal $E$ on $X$ with
$\nu^*E\simeq\mathbf 1_{\mathbb P^1}$.

The descent scalar for $E^{\otimes m}$ is $\lambda^m$.  Changing a
trivialization on $\mathbb P^1$ multiplies it by a single element of
$K^\times$, which does not change this scalar.  Since
$\lambda^m\ne 1$ for $m>0$, no positive tensor power of $E$ is
trivial.  Since $E$ has rank one, its monodromy group is a closed
subgroup of $\mathbb G_m$.  A finite such subgroup would be killed by
a positive power character and would make a positive tensor power of
$E$ trivial.  The monodromy group is therefore infinite, hence equal
to $\mathbb G_m$.  Thus $E$ is
not potentially isotrivial, although its pullback by the finite surjection
$\nu$ is trivial.  The geometrically unibranch hypothesis is therefore
essential, even when $X$, $Y$, and the covering morphism are proper.
\end{ex}

\section*{Declaration on the use of generative AI}

The authors first proved Theorem~\ref{main} in the smooth projective
case in July 2026, without the use of generative AI.
In September 2026, to assess the capabilities of generative AI, they
asked ChatGPT (OpenAI; GPT-5.6 Sol and GPT-6 Astra) to prove the theorem in full
generality. It produced a simpler proof that also applies in the
general case and forms the basis of the present paper.
The authors independently checked all mathematical arguments and
references, rewrote the generated material, and take full
responsibility for the contents of the article.

\section*{Acknowledgements}
The first author would like to thank  Sun Yat-Sen University for its hospitality during his stay in July 2026.
The first author was partially supported by National Science Centre, Poland (NCN),
contract number 2025/59/B/ST1/02168.  The second author was partially supported by the General Program of the National Natural Science Foundation of China (Grant No.~12671059) and the Guangdong Basic and Applied Basic Research Foundation (Grant No.~2025A1515012175).
\printbibliography
\end{document}